\documentclass[11pt,a4paper,headinclude,footinclude,fleqn]{scrartcl}                 
\usepackage[T1]{fontenc}                   
\usepackage[utf8]{inputenc}                 
\usepackage[english]{babel}       
\usepackage{graphicx}                      %
\usepackage[font=small]{quoting}            %
\usepackage{caption}                  %
\usepackage[nochapters,beramono,eulermath,%
            pdfspacing,
            listings,
            ]{classicthesis}                %
\usepackage{arsclassica}                    
\usepackage[top=.8in,bottom=1.2in,left=1.2in,right=1.2in]{geometry}
\usepackage[english]{babel}
\usepackage{verbatim}
\usepackage{color}
\usepackage{picture}
\usepackage{amsmath,amsthm,amsfonts,amssymb}
\usepackage{comment}
\usepackage{mathtools}

\usepackage{titling} 
\usepackage[hyperpageref]{backref}

\newtheorem{thm}{Theorem}[section]
\newtheorem{lem}[thm]{Lemma}

\newtheorem{prop}[thm]{Proposition}

\theoremstyle{definition}
\newtheorem{rem}[thm]{Remark}
\theoremstyle{remark}

\newcommand{\ds}{\displaystyle}

\newcommand{\R}{\mathbb{R}}
\newcommand{\N}{\mathbb{N}}

\newcommand{\de}{\partial}
\newcommand{\eps}{\varepsilon}

\newenvironment{sistema}%
{\left\{\begin{array}{@{}l@{}}}{\end{array}\right.}
\patchcmd{\abstract}{\scshape\abstractname}{\textbf{\abstractname}}{}{}
\makeatletter 
\def\@makefnmark{} 
\makeatother 
\title{A saturation phenomenon\\ for a nonlinear nonlocal eigenvalue problem}
\author{Francesco Della Pietra%
\thanks{Universit\`a degli studi di Napoli Federico II, Dipartimento di Matematica e Applicazioni ``R. Caccioppoli'', Via Cintia, Monte S. Angelo - 80126 Napoli, Italia. Email: f.dellapietra@unina.it} 
{, }\\
Gianpaolo Piscitelli%
\thanks{Universit\`a degli studi di Napoli Federico II, Dipartimento di Matematica e Applicazioni ``R. Caccioppoli'', Via Cintia, Monte S. Angelo - 80126 Napoli, Italia. Email: gianpaolo.piscitelli@unina.it}
}

\begin{document}

\maketitle


\begin{abstract}
Given $1\le q \le 2$ and $\alpha\in\R$, we study the properties of the solutions of the minimum problem 
\[
\lambda(\alpha,q)=\min\left\{\dfrac{\ds\int_{-1}^{1}|u'|^{2}dx+\alpha\left|\int_{-1}^{1}|u|^{q-1}u\, dx\right|^{\frac2q}}{\ds\int_{-1}^{1}|u|^{2}dx}, u\in H_{0}^{1}(-1,1),\,u\not\equiv 0\right\}.
\]
In particular, depending on $\alpha$ and $q$, we show that the minimizers have constant sign up to a critical value of $\alpha=\alpha_{q}$, and when $\alpha>\alpha_{q}$ the minimizers are odd.\\

\end{abstract}

\section{Introduction}
In this paper we consider the following problem:
\begin{equation}\label{operat}
\lambda(\alpha,q)=\inf \left\{ \mathcal Q[u,\alpha],\; u\in H_0^1(-1,1),\,u\not\equiv 0 \right\},
\end{equation}
where $\alpha\in\R$, $1\le q \le 2$ and
\[
\mathcal{Q}[u,\alpha]:=\dfrac{\ds\int_{-1}^{1}|u'|^{2}dx+\alpha\left|\int_{-1}^{1}|u|^{q-1}u\, dx\right|^{\frac2q}}{\ds\int_{-1}^{1}|u|^{2}dx}.
\]
Let us observe that $\lambda(\alpha,q)$ is the optimal value in 
the inequality
\begin{equation*}
\lambda(\alpha,q)\int_{-1}^{1}|u|^{2}dx\le{\ds\int_{-1}^{1}|u'|^{2}dx+\alpha\left|\int_{-1}^{1}|u|^{q-1}u\, dx\right|^{\frac2q}}{\ds}.
\end{equation*}
which holds for any $u\in H_0^1(-1,1)$.
Moreover, in the local case ($\alpha=0$), this inequality reduces to the classical one-dimensional Poincar\'e inequality; in particular, 
\[
\lambda(0,q)=\frac{\pi^2}{4}
\]
for any $q$.

The minimization problem \eqref{operat} leads, in general, to a nonlinear nonlocal eigenvalue problem. Indeed, supposing $\int_{-1}^1 y|y|^{q-1}\  dx\ge 0$, the associated Euler-Lagrange equation is
\begin{equation}
\left\{
\begin{array}{ll}
-y'' + \alpha\left(\ds\int_{-1}^1 y|y|^{q-1}\  dx\right)^{\frac{2}{q}-1} |y|^{q-1}=\lambda(\alpha,q)\, y& \text{in}\ ]-1,1[\\[.4cm]
y(-1)=y(1)=0
\end{array}
\right.
\end{equation}
(see Section 2 for its precise statement).

 This kind of nonlocal problems, in the one dimensional case, as in the multidimensional one, have been studied by several authors in many contexts, as, for example, reaction-diffusion equations describing chemical processes (see \cite{F}, \cite{S}) or Brownian motion with random jumps (see \cite{P}). Other results can be found, for example, in \cite{AB}, \cite{BCGM}, \cite{BFNT}, \cite{D}, \cite{F0}, \cite{Pi}.

The purpose of this paper is to study some properties of $\lambda(\alpha,q)$. In particular, depending on $\alpha$ and $q$, we aim to prove symmetry results for the minimizers of \eqref{operat}.

Under this point of view, in the multidimensional case ($N\ge 2$) the problem has been settled out in \cite{BFNT} (when $q=1$) and in \cite{D} (when $q=2$).

Our main result is stated in Theorem \ref{mainthm1} below. In particular, the nonlocal term affects the minimizer of problem \eqref{operat} in the sense that it has constant sign up to a critical value of $\alpha$ and, for $\alpha$ larger than the critical value, it has to change sign, and a saturation effect occurs.
\begin{thm}
\label{mainthm1}
Let $1\le q \le 2$. There exists a positive number $\alpha_{q}$ such that:
\begin{enumerate}
\item if $\alpha<\alpha_{q}$, then
\[
\lambda(\alpha,q)<\pi^{2},
\]
and any minimizer $y$ of $\lambda(\alpha,q)$ has constant sign in $]-1,1[$.
\item If $\alpha\ge\alpha_{q}$, then
\[
\lambda(\alpha,q)= \pi^{2}.
\]
Moreover, if $\alpha>\alpha_{q}$, the function {$y(x)=\sin\pi x$, $x\in[-1,1]$, is the only minimizer, up to a multiplicative constant, of $\lambda(\alpha,q)$. Hence it is odd, $\int_{-1}^{1} |y(x)|^{q-1}y(x)\,dx=0$, and $\overline x=0$ is the only point in $]-1,1[$ such that $y(\overline x)=0$.
}
\end{enumerate}
\end{thm}
Some additional informations are given in the next result.
\begin{thm}
\label{mainthm2}
The following facts hold.
\begin{enumerate}
\item For $q=1$, then $\alpha_{1}=\frac{\pi^{2}}{2}$. Moreover, if $\alpha=\alpha_{1}$, there exists a positive minimizer of $\lambda(\alpha_1,1)$, and for any $\overline x\in]-1,1[$ there exists a minimizer $y$ of $\lambda(\alpha_1,1)$ which changes sign in $\overline x$, non-symmetric and with $\int_{-1}^{1}y(x)\,dx\ne 0$ when $\overline x\ne 0$.
\item If $1<q\le 2$ and $\alpha=\alpha_{q}$, {then $\lambda(\alpha_q,q)$ in $[-1,1]$ admits both a positive minimizer and the minimizer $y(x)=\sin\pi x$, up to a multiplicative constant. Hence, any minimizer has constant sign or it is odd.}
\item If $q=2$, then $\alpha_{2}=\frac{3}{4}\pi^{2}$.
\end{enumerate}
\end{thm}
 {
Let us observe that, for any $\alpha\in\R$, it holds that
\begin{equation}
\label{nonloc-twist}
\lambda(\alpha,q) \le \Lambda_q=\pi^2,
\end{equation}
where
\begin{equation}
\label{twist}
\Lambda_q:=\min \left\{ \dfrac{\ds\int_{-1}^{1}|u'|^{2}dx}{\ds\int_{-1}^{1}|u|^{2}dx},\; u\in H_0^1(-1,1),\,\int_{-1}^{1}|u|^{q-1}u\,dx=0,\,u\not\equiv 0 \right\}.
\end{equation}
It is known that, when $q\in[1,2]$, then $\Lambda_q=\Lambda_1=\pi^{2}$, and the minimizer of \eqref{twist} is, up to a multiplicative constant, $y(x)=\sin \pi x$, $x\in[-1,1]$ (see for example \cite{CD}).} 

Problems with prescribed averages of $u$ and boundary value conditions have been studied in several papers. We refer the reader, for example, to \cite{BK,BKN,CD,E,EK,GN,N}. In recent literature, also the multidimensional case has been adressed (see, for example \cite{BDNT}, \cite{FH}, \cite{CHP,CHPerratum}, \cite{Narxiv}).

\begin{rem}
If the interval of integration is $]a,b[$ instead of $]-1,1[$, then
\[
\lambda(\alpha,q;]a,b[)= \left(\frac{2}{b-a}\right)^{2} \cdot 
\lambda\left(\left(\frac{b-a}{2}\right)^{1+\frac 2q}\alpha,\,q\right).
\]
\end{rem}

The outline of the paper follows. In Section 2 we show some properties of $\lambda(\alpha,q)$, while in Section 3 we study the behavior of changing-sign minimizers. Finally, in Section 4 we give the proof of the main results.

\section{Some properties of the first eigenvalue}
Let us observe that if $y$ is a minimizer in \eqref{operat}, then is not restrictive to suppose that $\int_{-1}^{1}|y|^{q-1}y\,dx\ge 0$. From now on, we assume that this condition is satisfied by the minimizers.
\begin{prop}
\label{propr}
Let $1\le q \le 2$ and $\alpha\in \R$. The following properties of $\lambda(\alpha,q)$ hold.
\begin{enumerate}
\item[(a)] Problem \eqref{operat} admits a minimizer.
\item[(b)] {Any minimizer $y$ of \eqref{operat} satisfies the following boundary value problem
\begin{equation}
\label{el}
\left\{
\begin{array}{ll}
-y'' + \alpha \gamma |y|^{q-1}=\lambda(\alpha,q)\, y& \text{in}\ ]-1,1[\\[.4cm]
y(-1)=y(1)=0,
\end{array}
\right.
\end{equation}
where 
\[
\gamma=
\begin{cases}
0 &\text{if both } q=2 \text{ and }\displaystyle \int_{-1}^1 y|y|\  dx=0, \\
\displaystyle\left(\int_{-1}^1 y|y|^{q-1}\  dx\right)^{\frac{2}{q}-1} &\text{otherwise}.
\end{cases}
\]
Moreover, $y\in C^{2}(-1,1)$. 
}
\item[(c)] 
 For any $q\in[1,2]$, the function $\lambda(\cdot,q)$ is Lipschitz continuous and non-decreasing with respect to $\alpha\in\R$. 
\item[(d)] If $\alpha\le 0$, the minimizers of \eqref{operat} do not change sign in $]-1,1[$. In addition, 
\[
\ds\lim_{\alpha\to -\infty}\lambda(\alpha,q)=-\infty.
\]
\item[(e)] As $\alpha\to+\infty$, we have that
{\[
\lambda(\alpha,q)\to \Lambda_q=\pi^{2}.
\]}
\end{enumerate}
\end{prop}
\begin{proof}
The existence of a minimizer follows immediately by the standard methods of Calculus of Variations. Furthermore, any minimizer satisfies \eqref{el}. In particular, let us explicitly observe that if $1\le q\le 2$ and there exists a minimizer $y$ of $\lambda(\alpha,q)$ such that $\int_{-1}^{1}|y|^{q-1}y\,dx=0$, then 
it holds that $\gamma=0$ in \eqref{el}. Indeed, in such a case $y$ is a minimizer also of the problem \eqref{twist}, whose Euler-Lagrange equation is (see \cite{CD})
\begin{equation*}
\left\{
\begin{array}{ll}
-y'' =\lambda(\alpha,q)\, y\quad &\text{in}\ ]-1,1[,\\[.4cm]
y(-1)=y(1)=0.
\end{array}
\right.
\end{equation*}
From \eqref{el} immediately follows that any minimizer $y$ is $C^{2}(-1,1)$. Hence {\it (a)-(b)} have been proved. 

In order to get property {\it (c)}, we stress that for all $\eps>0$, by H\"older inequality, it holds
\begin{equation*}
\mathcal{Q} [u,\alpha+\eps]\le\mathcal{Q}[u,\alpha]+\eps\frac{\left(\ds\int_{-1}^{1} |u|^q\  dx\right)^{2/q}}{\ds\int_{-1}^{1} u^2\  dx}\leq\mathcal{Q}[u,\alpha]+ 2^\frac{2-q}{q}\eps,\quad\forall\, \eps>0.
\end{equation*}
Therefore the following chain of inequalities
\begin{equation*}
\mathcal{Q}[u,\alpha]\leq\mathcal{Q}[u,\alpha+\eps]\le\mathcal{Q}[u,\alpha]+ 2^\frac{2-q}{q}\eps,\quad\forall \ \varepsilon>0,
\end{equation*}
implies, taking the minimum as $u\in H_0^1(-1,1)$, that
\begin{equation*}
\lambda (\alpha ,q)\leq\lambda (\alpha+\varepsilon,q)\leq\lambda (\alpha,q)+2^\frac{2-q}{q}\varepsilon,\quad\forall \ \varepsilon>0,
\end{equation*}
that proves {\it (c)}. 

If $\alpha< 0$, then
\[
\mathcal{Q}[u, \alpha] \ge \mathcal Q[|u|,\alpha],
\]
with equality if and only if $u\ge 0$ or $u\le 0$. Hence any minimizer has constant sign in $]-1,1[$. Finally, it is clear from the definition that $\ds\lim_{\alpha\to-\infty}\lambda(\alpha,q)=-\infty$. Indeed, by fixing a positive test function $\varphi$ we get
\[
\lambda(\alpha,q) \le \mathcal Q[\varphi,\alpha].
\]
Being $\varphi>0$ in $]-1,1[$, then  $\mathcal Q[\varphi,\alpha] \to -\infty \quad\text{as }\alpha\to -\infty$,
 and the proof of {\it (d)} is completed.

In order to show {\it (e)}, we recall that {$\lambda(\alpha,q)\le  \Lambda_q=\pi^{2}$. }

Let $\alpha_k\ge 0$, $k_n\in\N$, be a positively divergent sequence. For any $k$, consider a minimizer $y_k\in W_0^{1,2}$ of (\ref{operat}) such that $\|y_k\|_{L^2}=1$. We have that{
\begin{equation*}
\lambda(\alpha_k,q)=\int_{-1}^{1} |y'_k|^2\  dx + \alpha_k\left(\ds\int_{-1}^{1} y_k|y_k|^{q-1} \  dx\right)^\frac{2}{q}\leq \Lambda_q.
\end{equation*}}
Then $y_k$ converges (up to a subsequence) to a function $y\in H_0^1$, strongly in $L^2$ and weakly in $H_0^1$. Moreover $\|y\|_{L^2}=1$ and {
\begin{equation*}
\left( \int_{-1}^{1} y_k|y_k|^{q-1} \  dx\right)^\frac{2}{q}\leq\frac{ \Lambda_q}{\alpha_k}\rightarrow 0 \quad\text{as}\ k\rightarrow + \infty
\end{equation*}}
which gives that $\int_\Omega y|y|^{q-1} \  dx=0$. On the other hand the weak convergence in $H_0^{1}$ implies that
\begin{equation}
\label{lsclap}
\int_{-1}^{1} |y'|^2\  dx \leq \liminf_{k \rightarrow \infty}\int_{-1}^{1} |y'_k|^2\  dx.
\end{equation}
By definitions of {$\Lambda_q$ and $\lambda(\alpha,q)$, and by (\ref{lsclap}) we have
\begin{equation*}
\begin{split}
 \Lambda_q\le
\int_{-1}^{1} |y'|^2\  dx &\leq \liminf_{k \rightarrow \infty}\left[\int_{-1}^{1} |y'_k|^2\  dx
+ \alpha_k\left( \int_{-1}^{1} y_k|y_k|^{q-1} \  dx\right)^\frac{2}{q}\right]\\
&\leq\lim_{k \rightarrow \infty}\lambda(\alpha_k,q)\leq \Lambda_q.
\end{split}
\end{equation*}}
and the result follows.
\end{proof}
\begin{rem}
If $\lambda(\alpha,q)=0$, then
\begin{equation}
\label{lapuq2}
-\alpha=\min_{w\in H_0^1(-1,1)}\frac{\ds\int_{-1}^{1} |w'|^2 dx}{\left(\ds\int_{-1}^{1} |w|^{q} \ {d}x\right)^{2/q}}.
\end{equation}
Indeed, if $\lambda(\alpha,q)=0$ then necessarily $\alpha<0$ and the minimizers of \eqref{operat} have constant sign.
Let $y\ge 0$ be a minimizer of \eqref{operat}, by definition we have
\begin{equation*}
0=\lambda(\alpha,q)=\frac{\ds\int_{-1}^{1} | y'|^2\ dx+\alpha \left(\int_{-1}^{1} y^{q}\ dx\right)^{\frac{2}{q}}}{\ds\int_{-1}^{1} \bar u^2\ dx}
\end{equation*}
and hence
\begin{equation}
\label{alpugu}
-\alpha=\dfrac{\ds\int_{-1}^{1} | y'|^2dx}{\ds\left(\int_{-1}^{1} y^{q} \ dx\right)^{2/q}}.
\end{equation}
If we denote by $v$ a minimizer of problem \eqref{lapuq2}, we have
\begin{equation*}
0\leq \int_{-1}^{1} | v'|^2\  dx+\alpha \left(\int_{-1}^{1} |v|^{q}\ {d}x\right)^{{2}/{q}}
\end{equation*}
and therefore
\begin{equation*}
-\alpha\leq\dfrac{\ds\int_{-1}^{1} | v'|^2dx}{\ds\left(\int_{-1}^{1} |v|^{q} \ {d}x\right)^{2/q}}=\min_{w\in H_0^1(-1,1)}\frac{\ds\int_{-1}^{1} |w'|^2 dx}{\left(\ds\int_{-1}^{1} |w|^{q} \ {d}x\right)^{2/q}}.
\end{equation*}
From (\ref{alpugu}) the result follows.
\end{rem}

\section{Changing-sign minimizers}
{
We first analyze the behavior of the minimizers of \eqref{operat}, by assuming that they have to change sign in $]-1,1[$. In this case, by Proposition \ref{propr} {\it (d)}, we may suppose that $\alpha>0$. Moreover, due the homogeneity of the problem, in all the section we will assume also that
\[
\max_{[-1,1]} y(x)=1,\quad \min_{[-1,1]} y(x)=-\bar m,\quad \bar m\in]0,1].
\]
It is always possible to reduce to this condition, by multiplying the solution by a constant if necessary.
\\
We split the list of the main properties in two propositions.
\begin{prop}
\label{cambiosegno0}
Let $1\le q \le 2$ and suppose that, for some $\alpha>0$, $\lambda(\alpha,q)$ admits a minimizer $y$ that changes sign in $[-1,1]$. Then the following properties hold.
\begin{enumerate}
\item[(a1)] The minimizer $y$ has in $]-1,1[$ exactly one maximum point, $\eta_{M}$, with $y(\eta_{M})=1$, and exactly one minimum point, $\eta_{\bar m}$, with $y(\eta_{\bar m})= - \bar m$.
\item[(b1)] If $y_{+}\ge0$ and $y_{-}\le 0$ are, respectively, the positive and negative part of $y$, then $y_{+}$ and $y_{-}$ are, respectively, symmetric about $x=\eta_{M}$ and $x=\eta_{\bar m}$.
\item[(c1)] There exists a unique zero of $y$ in $]-1,1[$.
\item[(d1)]  In the minimum value $\bar m$ of $y$, it holds that  
\[
\lambda(\alpha,q)=H(\bar m,q)^2,
\] 
where $H(m,q)$, $(m,q)\in[0,1]\times[1,2]$, is the function defined as
\begin{multline*}
H(m,q):= \int_{-m}^1\frac{dy}{ \sqrt{1-z(m,q)(1- |y|^{q-1}y) - y^2}}=\\[.3cm]
=\int_{0}^1\frac{dy}{ \sqrt{1-z(m,q)(1-y^{q}) - y^2}} + \int_{0}^1\frac{mdy}{ \sqrt{1-z(m,q)(1+m^q y^{q}) -m^2 y^2}}
\end{multline*}
and $z(m,q)=\frac{1-m^2}{1+m^q}$.
\end{enumerate}
\end{prop}
\begin{prop}\label{cambiosegno} Let us suppose that, for some $\alpha>0$, $\lambda(\alpha,q)$ admits a minimizer $y$ that changes sign in $[-1,1]$. Then the following properties holds.
\item[(a2)] If $1\le q \le 2$, then 
\[
\lambda(\alpha,q)=\lambda^{T}=\pi^{2}.
\] 
\item[(b2)] If $1<q\le 2$, then
\begin{equation}
\label{q-average}
\int_{-1}^{1}|y|^{q-1}y\,dx=0.
\end{equation}
\item[(c2)] If $1\le q \le 2$ and \eqref{q-average} holds, then {$y(x)=C\sin \pi x$, with $C\in \R\setminus\{0\}$}. Hence the only point in $]-1,1[$ where $y$ vanishes is $\overline x=0$. 
\end{prop}
}
\begin{proof}[Proof of Proposition \ref{cambiosegno0}]
First of all, if $y$ is a minimizer of \eqref{operat} which changes sign, let us consider $\eta_{M},\eta_{\bar m}$ in $]-1,1[$ such that $y(\eta_{M})=1=\max_{[-1,1]}y$, and  $y(\eta_{\bar m})=-\bar m=\min_{[-1,1]}y$, with $\bar m\in]0,1]$. For the sake of simplicity, we will write $\lambda=\lambda(\alpha,q)$.
Multiplying the equation in \eqref{el} by $y'$ and integrating we get
\begin{equation}
\label{inel1d}
\frac{y'^2}{2}+\lambda \frac{y^2}{2} = \frac{\alpha\gamma}{q} |y|^{q-1}y + c\qquad\ \text{in } ] -1,1[,
\end{equation}
for a suitable constant $c$. Being $y'(\eta_{M})=0$ and $y(\eta_{M})=1$, we have
\begin{equation}
\label{cosmax}
c=\frac{\lambda}{2}-\frac{\alpha}{q} \gamma.
\end{equation}
Moreover, $y'(\eta_{\bar m})=0$ and $y(\eta_{\bar m})$ give also that
\begin{equation}
\label{cosmin}
c=\lambda\frac{\bar m^2}{2}+\frac{\alpha}{q} \bar m^q\gamma.
\end{equation}
Joining \eqref{cosmax} and \eqref{cosmin}, we obtain
\begin{equation}
\label{costnl}
\begin{sistema}
 \gamma =\frac{\ds q\lambda}{\ds 2 \alpha}z(\bar m,q)\\[.3cm]
c=\frac{\ds\lambda}{\ds 2}t(\bar m,q)
\end{sistema}
\end{equation}
where
\begin{equation*}
z(m,q)=\frac{1-m^2}{1+m^q} \quad\text{and}\quad t(m,q)=\frac{m^2+m^q}{1+m^q}=1-z(m,q).
\end{equation*}
Then \eqref{inel1d} can be written as 
\begin{equation}
\label{integratedel0}
\frac{y'^2}{2}+\lambda \frac{y^2}{2} = \frac{\lambda}{2} z(\bar m,q) |y|^{q-1}y + \frac{\lambda}{2} t(\bar m,q)\qquad\ \text{in }  ]-1,1[.
\end{equation}
Therefore we have
\begin{equation}
\label{integratedel}
(y')^2=\lambda [1-z(\bar m,q)(1- |y|^{q-1}y) - y^2] \qquad\ \text{in }  ]-1,1[.
\end{equation}

First of all, it is easy to see that the number of zeros of $y$ has to be finite.  Let 
\[
-1=\zeta_{1}<\ldots<\zeta_{j}<\zeta_{j+1}<\ldots<\zeta_{n}=1
\] 
be the zeros of $y$.

As observed in \cite{CD}, it is easy to show that 
\begin{equation}
\label{dac}
y'(x)=0 \iff y(x)=-\bar m\text{ or }y(x)=1.
\end{equation}
This implies that $y$ has no other local minima or maxima in $]-1,1[$, and in any interval $]\zeta_{j},\zeta_{j+1}[$ where $y>0$ there is a unique maximum point, and in any interval $]\zeta_{j},\zeta_{j+1}[$ where $y<0$ there is a unique minimum point. 

To prove \eqref{dac}, let
\[
g(Y)=1-z(\bar m,q)(1-|Y|^{q-1}Y)-Y^{2},\quad Y\in [-\bar m,1].
\]
So we have
\begin{equation}
\label{CDproof}
(y')^{2}=\lambda\, g(y).
\end{equation}
Observe that $g(-\bar m)=g(1)=0$. Being $q\le 2$, it holds that $g'(\bar Y)=0$ implies $g(\bar Y)>0$. Hence, $g$ does not vanish in $]-\bar m,1[$. By \eqref{CDproof}, it holds that $y'(x)\ne 0$ if $y(x)\ne 1 $ and $y(x)\ne -\bar m$. 

Now, adapting an argument contained in \cite[Lemma 2.6]{DGS}, the following three claims below allow to complete the proof of {\it (a1)}, {\it (b1)} and {\it (c1)}.

{
\begin{description}
\item[Claim 1:] in any interval $]\zeta_{j},\zeta_{j+1}[$ given by two subsequent zeros of $y$ and in which $y=y^{+}>0$, has the same length; in any of such intervals, $y^{+}$ is symmetric about $x=\frac{\zeta_{j}+\zeta_{j+1}}{2}$;
\item[Claim 2:] in any interval $]\zeta_{j},\zeta_{j+1}[$ given by two subsequent zeros of $y$ and in which $y=y^{-}<0$ has the same length; in any of such intervals, $y^{-}$ is symmetric about $x=\frac{\zeta_{j}+\zeta_{j+1}}{2}$;
\item[Claim 3:] there is a unique zero of $y$ in $]-1,1[$.
\end{description}
Then, $y$ admits a unique maximum point and a unique minimum point in $]-1,1[$, and the positive and negative parts are symmetric with respect to $x=\eta_{M}$ and $x=\eta_{\bar m}$, respectively. 

In order to get claims 1 and 2,  To fix the ideas, let us assume that $y>0$ in $]\zeta_{2k-1},\zeta_{2k}[$, and $y<0$ in $]\zeta_{2k},\zeta_{2k+1}[$. If this is not the case, the procedure is analogous.

Let us consider $y=y_{+}>0$ in $]\zeta_{2k-1},\zeta_{2k}[$, and denote by $\eta_{2k-1}$ the unique maximum point in such interval. Then by \eqref{integratedel}, integrating between $\zeta_{2k-1}$ and $\eta_{2k-1}$ we have
\[
\int_{0}^{1}\frac{dy}{\sqrt{1-z(\bar m,q)(1-y^q) - y^2}}
=(\eta_{2k-1}-\zeta_{2k-1})\sqrt{\lambda}.
\]
Similarly, integrating between $\eta_{2k-1}$ and $\zeta_{2k}$, it holds 
\[
\int_{0}^{1}\frac{dy}{\sqrt{1-z(\bar m,q)(1-y^q) - y^2}}
=(\zeta_{2k}-\eta_{2k-1})\sqrt{\lambda}.
\]
Hence 
\[
\zeta_{2k}-\zeta_{2k-1}=
\frac{1}{\sqrt{\lambda}}\int_{0}^{1}\frac{dy}{\sqrt{1-z(\bar m,q)(1-y^q) - y^2}},\quad\text{and}\quad\eta_{2k-1}=\frac{\zeta_{2k-1}+\zeta_{2k}}{2}.
\]
Similarly, consider that in $]\zeta_{2k},\zeta_{2k+1}[$ it holds $y=y_{-}<0$, and $y(\eta_{2k})=-\bar m$. By \eqref{integratedel}, integrating between $\zeta_{2k}$ and $\eta_{2k}$ we have
\[
\int_{0}^{\bar m}\frac{dy}{\sqrt{1-z(\bar m,q)(1+y^q) - y^2}}
=(\eta_{2k}-\zeta_{2k})\sqrt{\lambda},
\]
and then between $\eta_{2k}$ and $\zeta_{2k+1}$, it holds 
\[
\int_{0}^{\bar m}\frac{dy}{\sqrt{1-z(\bar m,q)(1+y^q) - y^2}}
=(\zeta_{2k+1}-\eta_{2k})\sqrt{\lambda}.
\]
Hence 
\[
\zeta_{2k+1}-\zeta_{2k}=
\frac{1}{\sqrt{\lambda}}\int_{0}^{\bar m}\frac{dy}{\sqrt{1-z(\bar m,q)(1+y^q) - y^2}},\quad\text{and}\quad\eta_{2k}=\frac{\zeta_{2k}+\zeta_{2k+1}}{2}.
\]
Resuming, we have that any interval given by two subsequent zeros of $y$ and in which $y=y_{+}>0$, has the same length. Similarly, 
any interval given by two subsequent zeros of $y$ and in which $y=y_{-}<0$, has the same length. 
\\
Now, again from \eqref{integratedel}, if $x\in]\zeta_{2k-1},\eta_{2k-1}[$ it holds
\begin{equation}
\label{sim1}
\int_{0}^{y(x)}\frac{dy}{\sqrt{1-z(\bar m,q)(1-y^q) - y^2}}
=(x-\zeta_{2k-1})\sqrt{\lambda},
\end{equation}
and, if $t\in]\eta_{2k-1},\zeta_{2k}[$, then
\[
-\int_{0}^{y(t)}\frac{dy}{\sqrt{1-z(\bar m,q)(1-y^q) - y^2}}
=(t-\zeta_{2k})\sqrt{\lambda}.
\]
On the other hand, by choosing $t=\zeta_{2k-1}+\zeta_{2k}-x\in]\eta_{2k-1},\zeta_{2k}[$ it holds
\[
-(x-\zeta_{2k-1})\sqrt{\lambda}=(t-\zeta_{2k})\sqrt{\lambda}=-\int_{0}^{y(t)}\frac{dy}{\sqrt{1-z(\bar m,q)(1-y^q) - y^2}}.
\]
From \eqref{sim1} we deduce that $y(x)=y(t)$, hence $y$ is symmetric about $x=\eta_{2k-1}$ in the interval $]\zeta_{2k-1},\zeta_{2k}[$.
\\
In the same way, $y$ is symmetric about $x=\eta_{2k}$ in the interval $]\zeta_{2k},\zeta_{2k+1}[$.
\\
Now we show that the number of the zeros of $y$ is odd. Let us observe that
\[
A_{+}:=\int_{\zeta_{2k-1}}^{\zeta_{2k}}y^{q}dx\ge
\int_{\zeta_{2k}}^{\zeta_{2k+1}}(-y)^{q}dx =:A_{-}\;.
\]
Indeed, multiplying \eqref{integratedel} by $|y(x)|^{2q}$ and using the symmetry properties of $y$ we have
\begin{multline*}
A_{+}= \frac{2}{\sqrt\lambda} \int_{\zeta_{2k-1}}^{\eta_{2k-1}}\frac{y^{q}}{\sqrt{1-z(\bar m,q)(1-y^q) - y^2}}y'dx
=\\=\frac{2}{\sqrt\lambda}\int_{0}^{1}\frac{y^{q}}{\sqrt{1-z(\bar m,q)(1-y^q) - y^2}}dy
\end{multline*}
and
\begin{multline*}
A_{-}= \frac{2}{\sqrt\lambda} \int_{\eta_{2k}}^{\zeta_{2k+1}}\frac{(-y)^{q}}{\sqrt{1-z(\bar m,q)(1+|y|^q) - y^2}}y'dx
=\\=\frac{2}{\sqrt\lambda}\int_{0}^{1}\frac{\bar m^{q+1}y^{q}}{\sqrt{1-z(\bar m,q)(1+\bar m^{q}y^q) - \bar m^{2}y^2}}dy.
\end{multline*}
If $y$ has an even number $n$ of zeros, two cases may occur.\\
{\bf Case 1: $\bar m=1$}. Then $z(1,q)=0$, and by \eqref{costnl} $\gamma=0$. On the other hand, $A_{+}=A_{-}$, and being $n$ even, then $\gamma=A_{+}^{\frac2q-1}$ and this is absurd.
\\
{\bf Case 2: $\bar m<1$.} Let us consider the function $\tilde y\in H^{1}_{0}(-1,1)$ defined as
\[
\tilde y(x)=
\begin{cases}
\phantom{-}y(x) &\text{if }x \in[\zeta_{0},\zeta_{n-1}]\\
-y(x) &\text{if }x \in [\zeta_{n-1},1].
\end{cases}
\]
We have that
\begin{equation}
\label{contra}
\begin{array}{l}
\ds\int_{-1}^{1}(\tilde y')^{2}dx =
\int_{-1}^{1}(y')^{2}dx,\; \int_{-1}^{1}\tilde y^{2}dx =
\int_{-1}^{1}y^{2}dx, \\[.4cm] \ds \left|\int_{-1}^{1}|\tilde y|^{q-1}\tilde y\,dx\right| < \int_{-1}^{1}|y|^{q-1}y\,dx.
\end{array}
\end{equation}
The first two equalities in \eqref{contra} are obvious. To show last inequality, we recall that $y(x)$ is positive in $]-1,\zeta_{2}[$, hence if it has an even number of zeros, it is positive in $]\zeta_{n-1},1[$.
Hence it is sufficient to observe that $A_{+} > A_{-}$ and
\[
\int_{-1}^{1}|y|^{q-1}y\,dx=\frac{n}{2} A_{+}-\frac{n-2}{2} A_{-},\qquad \int_{-1}^{1}|\tilde y|^{q-1}\tilde y\,dx=\frac{n-2}{2} A_{+}-\frac{n}{2}A_{-}.
\]
Then, \eqref{contra} implies that 
$\mathcal Q[\tilde y,\alpha]<\mathcal Q[ y,\alpha]$ and this contradicts  the minimality of $y$. So, the number $n$ of the zeros of $y$ is odd.
\\
Finally, we conclude that $n=3$ (Claim 3). If not, by considering the function $w(x)=y\left(\frac{2(x+1)}{n-1}-1\right)$, $x\in [-1,1]$, we obtain that 
\begin{equation*}
\mathcal{Q}[w,\alpha]=\dfrac{\left(\dfrac{2}{n-1}\right)^2 \ds\int_{-1}^{1}|y'|^{2}dx+\alpha\left|\int_{-1}^{1}|y|^{q-1}y\, dx\right|^{\frac2q}}{\ds \int_{-1}^{1}|y|^{2}dx}<\mathcal{Q}[y,\alpha],
\end{equation*}
that is absurd.
}
Hence, the solution $y$ has only one zero in $]-1,1[$, and also $(c1)$ is proved.

Now denote by $\eta_{M}$ and $\eta_{\bar m}$, respectively, the unique maximum and minimum point of $y$. It is not restrictive to suppose $\eta_{M}<\eta_{\bar m}$. They are such that $\eta_{M}-\eta_{\bar m}=1$, with $y'<0$ in $]\eta_{M},\eta_{\bar m}[$. Then
\begin{equation*}\sqrt{\lambda(\alpha,q)}=\frac{-y'}{\sqrt{1-z(\bar m,q)(1-|y|^{q-1}y)- y^2}} \qquad\ \text{in} \ ]\eta_{M},\eta_{\bar m}[.
\end{equation*}
Integrating between $\eta_{M}$ and $\eta_{\bar m}$, we have
\begin{equation*}
\lambda(\alpha,q)=\left[\int_{-\bar m}^1\frac{dy}{ \sqrt{1-z(\bar m,q)(1-|y|^{q-1}y) - y^2}}\right]^{2}= H(\bar m,q),
\end{equation*}
and the proof of the Proposition is completed.
\end{proof}
\begin{rem}
We stress that properties $(a1)-(a3)$ can be also proved by using a symmetrization argument, by rearranging the functions $y^{+}$ and $y^{-}$ and using the P\'olya-Sz\H{e}go inequality and the properties of rearrangements (see also, for example, \cite{BFNT} and \cite{D}). For the convenience of the reader, we prefer to give an elementary proof without using the symmetrization technique.
\end{rem}

Our aim now is to study the function $H$ defined in Proposition \ref{cambiosegno0}.
\begin{prop}
\label{propminimo}
For any $m\in[0,1]$ and $q\in [1,2]$ it holds that
\[
H(m,q) \ge H(m,1)=\pi.
\]
Moreover, if $m<1$ and $q>1$, then
\[
H(m,q)>\pi,
\]
while
\[
H(m,1)=\pi,\quad \forall m\in[0,1].
\]
Hence if $H(m,q)=\pi$ and $1<q\le 2$, then necessarily $m=1$.
\end{prop}
\begin{rem}
{%
Let us explicitly observe that in the case $\alpha\le 0$, it holds that
 $\lambda(\alpha,q)\le \frac{\pi^{2}}{4}<H(m,q)^{2}$ for any $m\in[0,1]$ and $q\in[1,2]$.
 }
\end{rem}
\begin{rem}
\label{rem3.5}
The proof of Proposition \ref{propminimo} is based on the study of the integrand function that defines $H(m,q)$, that is
\begin{equation*}
h(m,q,y):=\frac{1}{ \sqrt{1-z(m,q)(1-y^{q}) - y^2 }}
+\frac{m}{ \sqrt{1-z(m,q)(1+ m^{q}y^{q})- m^{2}y^2 }}.
\end{equation*}
Let us explicitly observe that if $m=1$, then $z(1,q)=0$ and
\[
h(1,q,y)=\frac{2}{ \sqrt{1- y^2 }}, 
\]
that is constant in $q$. Moreover, if $y=0$, then
\[
h(m,q,0)=\frac{1+m}{\sqrt{1-z(m,q) }}
\]
that is strictly increasing in $q\in[1,2]$. Furthermore, simple computations yield
\[
\left\{
\begin{array}{ll}
H(1,q)=\pi, & H(0,q)=\dfrac{\pi}{2-q}\ge \pi\quad(H(0,2)=+\infty), \\[.5cm]
H(m,1)=\pi, & H(m,2)=\ds\frac\pi2 \sqrt{\frac{1+m^{2}}{2}}\left(\frac 1m +1\right)\ge \pi.
\end{array}
\right.
\]
\end{rem}
To prove Proposition \ref{propminimo}, it is sufficient to show that $h$ is monotone in $q$.

\begin{lem}
\label{lemma-monotonia}
For any fixed $y\in[0,1[$ and $m\in]0,1[$, the function $h(m,\cdot,y)$ is strictly increasing as $q\in[1,2]$.
\end{lem}
\begin{proof}
From the preceding observations,  we may assume $m\in]0,1[$ and $y\in]0,1[$.
Differentiating in $q$, we have that
\begin{equation*}
\begin{split}
\de_{q}h = &- \frac {1}{2F_{I}^{3}} \big[-(1-y^{q})\de_{q} z + z\, y^{q}\log y \big]+\\[.2cm]
&- \frac {m}{2F_{II}^{3}} \big[
-(1+m^{q}y^{q})\de_{q}z- z\,m^{q}y^{q}(\log m +\log y)
\big],
\end{split}
\end{equation*}
where
\begin{equation}
\label{firint}
F_{I}(m,q,y):= \sqrt{1-z(m,q)(1- y^{q}) - y^2} \le \sqrt{1-y^2},
\end{equation}
and 
\begin{equation}
\label{secint}
F_{II}(m,q,y):=\sqrt{1-z(m,q)(1+m^{q}y^{q})-m^{2}y^{2}}\ge m{\sqrt{1-y^{2}}}.
\end{equation}
Being 
\[
z=\frac{1-m^{2}}{1+m^{q}},\quad
\de_{q}z= -\frac{1-m^{2}}{(1+m^{q})^{2}}m^{q}\log m,
\]
we have that
\begin{equation*}
\begin{split}
\de_{q}h = \frac12\frac{1-m^{2}}{(1+m^{q})^{2}}\bigg\{ &\overbrace{\bigg[ -(1-y^{q})m^{q}\log m -y^{q}(1+m^{q})\log y\bigg]}^{h_{1}(m,q,y)}\frac{1}{F_{I}^{3}}+ \\[.2cm]
+& \underbrace{\bigg[-(1+m^{q}y^{q})\log m+(1+m^{q})y^{q}(\log m +\log y)\bigg]}_{h_{2}(m,q,y)}\frac {m^{q+1}}{F_{II}^{3}}\bigg\}.
\end{split}
\end{equation*}

Let us observe that $h_{1}(m,q,y)\ge 0$. Hence, in the set $A$ of $(q,m,y)$ such that $h_{2}(m,q,y)$ is nonnegative, we have that $\de_{q} h(q,m,y) \ge 0$. Moreover, $h_{1}(q,m,y)$ cannot vanish ($y<1$), then $\de_{q}h>0$ in $A$.

Hence, let us consider the set $B$ where
\[
h_{2}=(y^{q}-1)\log m+(1+m^{q})y^{q}\log y\le 0
\] 
(observe that in general $A$ and $B$ are nonempty). 
By \eqref{firint} and \eqref{secint} we have that
\begin{equation*}
\begin{split}
\de_{q}h \ge \frac12\frac{1-m^{2}}{(1+m^{q})^{2}}\bigg\{ &\bigg[ -(1-y^{q})m^{q}\log m -y^{q}(1+m^{q})\log y\bigg]\frac{1}{(1-y^{2})^{\frac32}}+ \\[.2cm]
+& \bigg[(y^{q}-1)\log m+(1+m^{q})y^{q}\log y\bigg]\frac {m^{q-2}}{(1-y^{2})^{\frac32}}\bigg\}.
\end{split}
\end{equation*}
Hence, to show that $\de_{q}h> 0$ also in the set $B$ it is sufficient to prove that
\begin{multline}
\label{g>0}
g(m,q,y):=\bigg[ -(1-y^{q})m^{q}\log m -y^{q}(1+m^{q})\log y\bigg]+\\
+ \bigg[(y^{q}-1)\log m+(1+m^{q})y^{q}\log y\bigg]m^{q-2}> 0
\end{multline}
when $m\in]0,1[$, $q\in[1,2]$ and $y\in]0,1[$.\\

{\bf Claim 1.} {\em For any $q\in [1,2]$ and $m\in]0,1[$, the function $g(m,q,\cdot)$ is strictly decreasing for $y\in]0,1[$.} 
\\

To prove the Claim 1, we differentiate $g$ with respect to $y$, obtaining
\begin{multline*}
\de_{y}g= \bigg[qy^{q-1}m^{q}\log m-qy^{q-1}(1+m^{q})\log y-y^{q-1}(1+m^{q})\bigg]+\\
+\bigg[qy^{q-1}\log m+(1+m^{q})(qy^{q-1}\log y+y^{q-1})
\bigg]m^{q-2}=\\
=y^{q-1}\bigg[q (m^{q} +m^{q-2})\log m+q(1+m^{q})(m^{q-2}-1)\log y + (1+m^{q})(m^{q-2}-1)\bigg].
\end{multline*}
Then $\de_{y} g < 0$ if and only if
\begin{equation*}
(1+m^{q})(m^{q-2}-1)(q\log y+1) < -q(m^{q}+m^{q-2})\log m.
\end{equation*}
The above inequality is true, as we will show that (recall that $0<m<1$ and $1\le q\le 2$)
\begin{equation}
\label{ineqlog}
\log y < -\frac1q +\frac{(m^{q}+m^{q-2})\log m}{(1+m^{q})(1-m^{q-2})}=:-\frac1q+\ell(m,q). 
\end{equation}
If the the right-hand side of \eqref{ineqlog} is nonnegative, then for any $y\in ]0,1[$ the inequality \eqref{ineqlog} holds.\\

{\bf Claim 2.} {\em For any $q\in [1,2]$ and $m\in]0,1[$, $\ell(m,q)> \frac1q$.} \\

We will show that
\[
\ell(m,q) > 1 \ge \frac 1q.
\]
We have
\[
\ell(m,q)= \frac{(m^{q}+m^{q-2})}{(1+m^{q})(m^{q-2}-1)}\log\frac{1}{m} > 1
\]
if and only if
\begin{multline*}
\mu(m,q)=(m^{q}+m^{q-2})\log\frac{1}{m} - (1+m^{q})(m^{q-2}-1)=\\
=(m^{q}+m^{q-2})\log\frac{1}{m} + 1+m^{q}-m^{q-2}-m^{2q-2}=\\
= m^{q}\left(\log\frac 1m+1\right)+ m^{q-2}\left(\log\frac1m-1\right)+1-m^{2q-2}
 > 0.
\end{multline*}
Then for $m \in]0,1[$ we have
\begin{multline*}
\mu(m,q)=m^{q}\left(\log\frac 1m+1\right)+ m^{q-2}\left(\log\frac1m-1\right)+1-m^{2q-2}\\
\ge m^{q}\left(\log\frac 1m+1\right)+ m^{q-2}\left(\log\frac1m-1\right)=\\=
m^{q-2}\left(m^{2}\left(\log\frac 1m+1\right)+\log\frac 1m-1\right)> 0,
\end{multline*}
and the Claim 2, and then the Claim 1, are proved. To conclude the proof of \eqref{g>0}, it is sufficient to observe that 
\[
g(m,q,y)> g(m,q,1)=0
\]
when $m\in]0,1[$, $q\in[1,2]$ and $y\in]0,1[$.

The Claim 1 gives that $\de_{q}h(m,q,y)> 0$ when $m\in]0,1[$, $q\in[1,2]$ and $y\in]0,1[$, and this conclude the proof.

%
%
%

\end{proof}

\begin{proof}[Proof of Proposition \ref{propminimo}]
Using Lemma \ref{lemma-monotonia} and Remark \ref{rem3.5}, it holds that
\[
H(m,q) \ge H(m,1)=\pi
\]
for $1\le q \le 2$. In particular, if $q\in]1,2]$, $m\in[0,1[$ and $y\in ]0,1[$ then
\[
h(m,q,y)>h(m,1,y),
\]
hence for any $m\in[0,1[$ and $q\in]1,2]$ it holds
\[
H(m,q)>H(m,1)=\pi.
\]
\end{proof}
Now we are in position to prove Proposition \ref{cambiosegno}. 
\begin{proof}[Proof of Proposition \ref{cambiosegno}]
Let $y$ be a minimizer of $\lambda(\alpha,q)$ that changes sign in $[-1,1]$, with $\max_{x\in[-1,1]} y(x)=1$.
By  {\it (d1)} of Proposition \ref{cambiosegno0} and Proposition \ref{propminimo}, the eigenvalue $\lambda(\alpha,q)$ has to satisfy the inequality
\[
\lambda(\alpha,q)\ge \pi^{2}.
\]
Hence, by \eqref{nonloc-twist}  it follows that
\[
\lambda(\alpha,q) = \pi^{2},
\]
that is property {\it (a2)}.
Assuming also $1<q\le 2$, if $-\bar m$ is the minimum value of $y$, again by Proposition \ref{propminimo} and {\it (d1)} of Proposition \ref{cambiosegno0}, 
$\lambda(\alpha,q)=\pi^{2}$ if and only if $\bar m=1$. Hence, $z(1,q)=0$ and the first identity of \eqref{costnl} gives that 
\[
\int_{-1}^{1} y|y|^{q-1}dx=0.
\]
and hence {\it (b2)} follows.

To prove {\it (c2)}, let us explicitly observe that, when \eqref{q-average} holds, 
 $y$ solves
\[
\begin{cases}
y''+\pi^{2}y=0&\text{in }]-1,1[\\
y(-1)=y(1)=0.&\\
\end{cases}
\]
Hence {$y(x)=C\sin \pi x$, with $C\in \R\setminus\{0\}$}.
%
\end{proof}

\section{Proof of the main results}

Now we are in position to prove the first main result.
\begin{proof}[Proof of Theorem \ref{mainthm1} and Theorem \ref{mainthm2}]
We begin the proof with the following claim.\\

{\bf Claim.}
{\it There exists a positive value of $\alpha$ such that the minimum problem
\begin{equation*}
\lambda(\alpha,q)=\min_{u \in H_0^1([-1,1])}\frac{\ds \int_{-1}^{1} | u'|^2\  dx+\alpha \left|\int_{-1}^{1} u|u|^{q-1}\  dx\right|^{\frac{2}{q}}}{\ds\int_{-1}^{1} u^2\  dx}
\end{equation*}
admits an eigenfunction $y$ that satisfies $\int_{-1}^1y|y|^{q-1}\  dx=0$ In such a case, $\lambda(\alpha,q)=\pi^{2}$ and, up to a multiplicative constant, $y=\sin \pi x$.
} 

To prove the claim, let us consider the case $1<q\le 2$. By contradiction, we suppose that for any $k\in\N$, there exists a divergent sequence $\alpha_{k}$, and a corresponding sequence of eigenfunctions $\{y_{k}\}_{k\in\N}$ relative to $\lambda(\alpha_{k},q)$ such that $\int_{-1}^1y_{k}|y_{k}|^{q-1}dx>0$ and $\|y_k\|_{L^2(-1,1)}=1$. By Proposition \ref{cambiosegno}, these eigenfunctions do not change sign and, as we have already observed, $\lambda(\alpha_{k},q)\le\pi^{2}$. It holds that
\begin{equation}
\label{contrad}
\ds \int_{-1}^{1} | y_k'|^2\  dx+\alpha_k \left(\int_{-1}^{1} |y_k|^{q}\  dx\right)^{\frac{2}{q}}\le\pi^{2}.
\end{equation}
Hence, $y_k$ converges (up to a subsequence) to a function $y\in W_0^{1,2}(-1,1)$, strongly in $L^2(-1,1)$ and weakly in $H_{0}^{1}(-1,1)$. Moreover $\|y\|_{L^2(-1,1)}=1$ and $y$ is not identically zero. Hence $\|y\|_{L^q(-1,1)}>0$. Therefore, letting $\alpha_k\rightarrow +\infty$ in \eqref{contrad} we have a contradiction and the claim is proved.

Now, we recall that for any $1\le q \le2$, $\lambda(\alpha,q)$ is a nondecreasing Lipschitz function in $\alpha$, and for $\alpha$ sufficiently large, $\lambda(\alpha,q)=\pi^{2}$.  Hence, using the Claim 1, we can define 
\[
\alpha_{q}=\min\{\alpha\in\R\colon \lambda(\alpha,q)=\pi^{2}\}=\sup\{\alpha\in\R\colon \lambda(\alpha,q)<\pi^{2}\}.
\]

Obviously, $\alpha_{q}>0$. If $\alpha<\alpha_{q}$, then the minimizers corresponding to $\lambda(\alpha,q)$ has constant sign, otherwise $\lambda(\alpha,q)=\pi^{2}$. If $\alpha>\alpha_{q}$, then any minimizer $y$ corresponding to $\alpha$ is such that $\int_{-1}^{1}|y|^{q-1}y\,dx=0$. Indeed, if we assume, by contradiction, that there exist $\bar \alpha>\alpha_{q}$ and $\bar y$ such that $\int_{-1}^{1}|\bar y|^{q-1}\bar y\,dx>0$, $\|y\|_{L^2}=1$ and $\mathcal Q[\bar\alpha,\bar y]=\lambda(\bar\alpha,q)$, then 
\begin{align*}
\mathcal Q[\bar\alpha-\eps,\bar y] &=
\mathcal Q[\bar\alpha,\bar y]-\eps\left(\int_{-1}^{1}|\bar y|^{q-1}\bar y\,dx\right)^{\frac q 2}\\ &=\lambda(\bar\alpha,q)-\eps\left(\int_{-1}^{1}|\bar y|^{q-1}\bar y\,dx\right)^{\frac q 2}<\lambda(\bar\alpha,q).
\end{align*}
Hence, for $\eps$ sufficiently small, $\pi^{2}=\lambda(\alpha_{q},q)\le\lambda(\bar\alpha-\eps,q)<\lambda(\bar\alpha,q)$ and this is absurd.
{Finally, by (c2) of Proposition \ref{cambiosegno}, the proof of Theorem \ref{mainthm1} is completed. It is not difficult to see, by means of approximating sequences, that $\lambda(\alpha_{q},q)$ admits both a nonnegative minimizer and a minimizer with vanishing $q$-average, that gives the thesis of Theorem \ref{mainthm2}, in the case $1<q\le 2$. To conclude the proof of Theorem \ref{mainthm2}, we have to study the behavior of the solutions when $q=1$ and $q=2$.
Despite its simplicity, the case $q=1$ has a peculiar behavior. Let us recall that $\lambda(0,1)=\frac{\pi^{2}}{4}$, and, being $\lambda(\alpha,1)$ is Lipschitz, it assumes all the values in the interval $\left]\frac{\pi^{2}}{4},\pi^{2}\right]$ as $\alpha$ varies in $]0,+\infty[$.
}

Suppose that $\frac{\pi^{2}}{4}<\lambda(\alpha,1)<\pi^{2}$. Then $0<\alpha<\alpha_{1}$, the corresponding minimizer $y$ has constant sign in $]-1,1[$, and it is a solution of
\begin{equation*}
\left\{
\begin{array}{ll}
y''+\lambda y=\alpha \gamma &\text{in }]-1,1[\\
y(-1)=y(1)=0,
\end{array}
\right.
\end{equation*}
where $\lambda=\lambda(\alpha,1)$ and $\gamma=\ds\int_{-1}^{1}y(x)dx>0$. Hence
\[
y(x)=\frac{\gamma \alpha}{\lambda}\left(1-\frac{\cos(\sqrt \lambda x)}{\cos\sqrt\lambda}\right),\quad x\in[-1,1].
\]
Integrating both sides in $[-1,1]$, we get
\[
\alpha=\frac{\lambda\sqrt\lambda}{2\sqrt\lambda-2{\tan{\sqrt\lambda}}},
\]
and, letting $\lambda\to {\pi^{2}}$, 
\[
\alpha_{1}= \frac{\pi^{2}}{2}.
\]

Finally, in the critical case $\alpha=\alpha_{1}=\frac{\pi^{2}}{2}$,  an immediate computation shows that the functions
\[
y_{A}(x)= \frac{A}{2}\left( 1+\cos \left( \pi x \right)\right) -\sqrt{1-A}\sin \left( \pi x \right),
\]
with $A\in[0,1]$ have average $\gamma=A$ and $y_{A}$ are minimizers of $\lambda(\alpha_{1},1)=\pi^{2}$. Moreover, when $A$ varies in $[0,1[$, the root of $y_{A}$ in $]-1,1[$ varies continuously in $[0,1[$.

It remains to consider the case $q=2$. If $\alpha=\alpha_2$, the corresponding positive minimizer $y$ is a solution of
\begin{equation*}
\left\{
\begin{array}{ll}
y''+\pi^2 y=\alpha_2 y &\text{in }]-1,1[\\
y(-1)=y(1)=0.
\end{array}
\right.
\end{equation*}
 The positivity of the eigenfunction guarantees that 
\begin{equation*}
\alpha_2-\pi^2=\lambda(0,2)=\frac{\pi^2}{4},
\end{equation*}
hence $\alpha_{2}=\frac{3}{4}\pi^2$.
\end{proof}


\begin{rem}
When $1<q<2$, it is possible to obtain the following lower bound on $\alpha_q$:
	\begin{equation}
	\label{stimaq}
	\alpha_{q}\ge \frac{3}{2^{1+\frac 2q}}\pi^{2}.
	\end{equation}
	To get the estimate \eqref{stimaq}, by choosing $u(x)=\cos\frac\pi2x$ as test function we get
	\begin{equation*}
	\pi^{2}=\lambda(\alpha_{q},q)\le\mathcal{Q}[u,\alpha_{q}]=\frac{\pi^{2}}{4}+\alpha_{q}\left(\int_{-1}^{1}u^{q}dx\right)^{2/q}\le \frac{\pi^{2}}{4} +\alpha_{q}2^{\frac2q-1}.
	\end{equation*}
\end{rem}

\end{document}